\documentclass[12pt,a4paper]{article}

\usepackage{amssymb}
\usepackage{amsmath}
\usepackage{amsthm}
\usepackage{hyperref}
\usepackage{calrsfs}
\usepackage{bclogo}
\usepackage{amsmath}
\usepackage{amsfonts}
\usepackage{amssymb}
\usepackage{booktabs}
\usepackage{tabularx}
\usepackage[ruled,linesnumbered]{algorithm2e}
\usepackage{subcaption}
\usepackage{multicol}
\usepackage{enumerate}
\usepackage{url}
\usepackage{float}
\usepackage{fix-cm}
\usepackage{arcs}
\usepackage{etoolbox}
\usepackage{epsfig}
\usepackage[T1]{fontenc}
\usepackage{amsthm}
\usepackage{amssymb}
\usepackage{amsmath}
\usepackage{hyperref}
\usepackage{enumerate}
\usepackage{listings}
\usepackage{microtype}
\usepackage{stackengine}
\usepackage{comment}
\usepackage{array}

\newtheorem{thm}{Theorem}[section]

\newtheorem{problem}[thm]{Problem}

\theoremstyle{definition}

\begin{document}

\title{A convex $\sigma$-morphic protoset}

\author{Aleksa D\v zuklevski\\
\em \small Department of Mathematics and Informatics, University of Novi Sad,\\
\em \small Trg Dositeja Obradovi\'ca 4, 21000 Novi Sad, Serbia\\
\small \tt aleksa.dzuklevski@dmi.uns.ac.rs}
\date{}

\maketitle

\begin{abstract}
We say that a tile is $\sigma$-morphic if it tiles the plane in exactly $\aleph_0$ many noncongruent ways (up to an isometry). It is an unsolved problem of whether a $\sigma$-morphic tile exist in the plane. In this note we present a construction of a set of convex tiles that is $\sigma$-morphic. The result is interesting since all the constructions of $\sigma$-morphic sets of tiles that arise in the literature make use of bumps and nicks, which necessarily make the tiles non-convex. We construct our set by cleverly dividing the tiles of the set of tiles discovered by Schmitt into convex tiles so that they behave in the same manner.

\emph{Mathematics Subject Classification (2010):} 52C20, 05B45, 52A37

\emph{Keywords:} tiling, tessellation, polymorphic tile, $\sigma$-morphic tile, convex protosets
\end{abstract}

\section{Introduction}

The theory of tilings is a relatively young research area in discrete geometry and there are many yet unsolved problems concerning various elementary notions regarding tiles. Formally, the tiling $\mathcal{T}$ of the Euclidean plane is a set of elements $T_i$, $i\in \mathbb{N}$, which are closed topological discs in $\mathbb{E}^2$, such that $\bigcup_{i \in \mathbb{N}} T_i = \mathbb{E}^2$ and that $\textup{int}T_i \cap \textup{int}T_j = \varnothing$ for all $i \neq j.$ If all $T_i$ are congruent to some tile $T$, we say that $T$ tiles the plane, or that it admits a tiling of the plane. The set of distinct tiles that appear in a tiling is called the protoset and its elements the prototiles. Amongst the unsolved problems in tiling theory there are those that are more combinatorial in nature, and perhaps the best example is the problem of the so called \textit{polymorphic} tiles.

We say that a tile is $\textit{m-morphic}$ if it tiles the plane in exactly $m$ non-congruent ways, where two tilings are congruent if there is an isometry that maps one to the other \cite{TIP}. Gr{\"u}nbaum and Shephard asked in 1977 \cite{PDT} if for each integer $m$ there exists an $m$-morphic tile, and the problem is still open. The record holder is an $11$-morphic polyomino discovered by Myers \cite{myerspoly}.

Gr{\"u}nbaum and Shephard then focused on tiles that tile the plane in infinitely many ways and proceeded to make a fascinating observation: each tile they found, which tiles the plane in $\textit{infinitely}$ many ways, did so in $\mathfrak{c}$ many ways! A very natural question was then asked, of the existence of a tile that tiles the plane in exactly $\aleph_0$ many distinct ways, which was dubbed $\sigma$-morphic, and whose existence is still an open problem.

One of the main ideas behind "forcing" the tiles to behave in the way one would like them to, and which stands behind a lot of constructions seen in the literature (the construction of $\sigma$-morphic protosets by Schmitt \cite{schmitttrougao, schmitthierarchy} or of a $\sigma$-morphic tile with generalized matching rules by Ba\v si\' c, D\v zuklevski and Slivkov\' a \cite{mi}) is to use bump and nicks on the tiles, which then force specific edges of tiles to meet in a tiling (cf. Figure \ref{fig schmitt2}). Thus the problem of finding a protoset of convex prototiles that is $\sigma$-morphic is an interesting one. In this note we show how it can be constructed. 

The organization of the paper is as follows. In the next section we give an overview of the known results on $\sigma$-morphic tiles, the mentioned variations that have been solved, and survey the main ideas of forcing $\aleph_0$ many distinct tilings. In Section \ref{sec4}, we show the existence of a $\sigma$-morphic protoset which consists solely of convex prototiles. In the last section we give some open problems and directions for future research.

\section{An overview of the main ideas from the literature} \label{sec2}

\subsection{Different ways of achieving $\mathfrak{c}$ many tilings}\label{sec2.1}
Let us first consider how $\mathfrak{c}$ many tilings can arise. To that end, consider a regular tiling by unit squares and note that one can slide any of the rows in this tiling with respect to the rest for any real number between $0$ and $\frac{1}{2}$ to obtain a distinct tiling by squares. And indeed, one of the ways a $\mathfrak{c}$ many tilings can arise is by "shifting" of a part of the tiling, akin to the one just described. A more general way to look at these tilings is to say that there is a subset of the plane which can be tiled in $\mathfrak{c}$ many ways, and that can be extended into a tiling of the whole plane, all of which are distinct.

A lot of tiles, however, exhibit a more combinatorially rich behavior. To see this, observe the tile shown in Figure \ref{combcont} \cite{TIP}. Note that there are two "row-like" structures that this tile can make: the one in which the tiles are "pointing" towards the left (orange in Figure \ref{combcont}) and the one in which tiles are "pointing" to the right (ocher-greyish in Figure \ref{combcont}). Notice next how both the boundary of the orange row (blue in Figure \ref{combcont}) and the boundary of the ocher-greyish row (yellow in Figure \ref{combcont}) are the same. What this means is that the rows influence the tilings in the same way, regardless of the orientation of the tiles in them. It is then easy to observe that there is a bijective correspondence between tilings by this tile (at least those that consist of "rows" like the ones shown here) and bi-infinite sequences of letters $A$ and $B$, of which there are $2^{\aleph_0} = \mathfrak{c}$ many distinct ones.

\begin{figure}[h!]
\centering
\includegraphics[width=8cm]{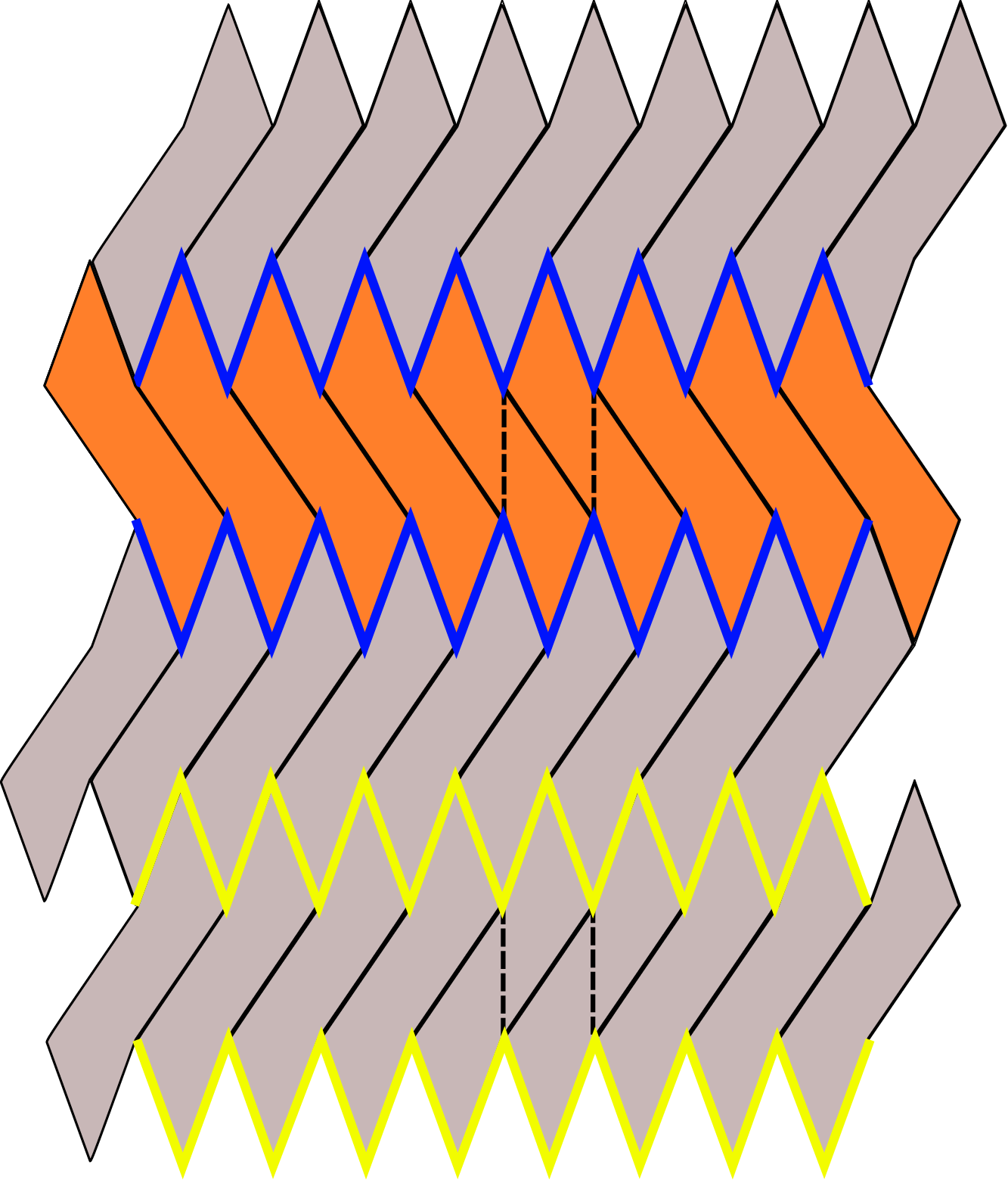}
\caption{ A $\mathfrak{c}$-morphic tile. Note that the blue outline is the same as the yellow one, allowing the orange tiles to be oriented in a different direction, without requiring any change in the rest of the tiling. This freedom of choice at every layer gives rise to $\mathfrak{c}$ many tilings. }\label{combcont}
\end{figure}

Almost all the so far found $\mathfrak{c}$-morphic tiles are enforcing these two structures and it is for this reason that the tiles shown in Figure \ref{fig myerspriv} are very interesting. They were found by Joseph S. Myers \cite{myerspriv}, and exhibit a hierarchical structure to their tilings. The hierarchy is built of larger and larger nested triangular structures, such as the one shown in Figure \ref{fig hierarchy}. To see why this structure implies the existence of $\mathfrak{c}$ many tilings, observe first that each triangle (of any size - for instance, the smallest one,  which is any three incident tiles colored in cyan in Figure \ref{fig hierarchy}) can be any of the four sub-triangles of a twice larger triangle: it can be in the top, middle, left or right, which we will denote by positions $1,2,3$ and $4$, in that order. The way in which we obtain a tiling of the plane is by extending this triangle into ever larger and larger triangles, with which we associate an infinite sequence of numbers $1,2,3,4$ in which each of $1,3,4$ appears infinitely often, of which there are $\mathfrak{c}$ many. For instance, the sequence $1,2,3,...$ would mean that the smallest triangle $T_1$ is at the corner $1$ of the larger triangle $T_2,$ $T_2$ is in the center of the triangle $T_3,$ which is in the corner $3$ of the larger triangle $T_4$ and so on. For a given hierarchical tiling, and a given starting triangle (comprising of $3$ incident cyan tiles), the sequence of nestings starting with that triangle and which result in the given tiling is unique and can be discerned from the tiling. As there are only countably many potential starting triangles in any tiling, each tiling will be counted in this way only by a countable number of allowed sequences, but seeing how there are $\mathfrak{c}$ many tilings in total (as they correspond to the sequences with infinitely many elements $1,3$ and $4$), the number of distinct tilings (which we obtain by "dividing" all the tilings by the number of times they each appear in this counting scheme), is still uncountable.

\begin{figure}[h!]
        
        \begin{subfigure}[b]{0.5\textwidth}
            \centering
            \includegraphics[width=5cm]{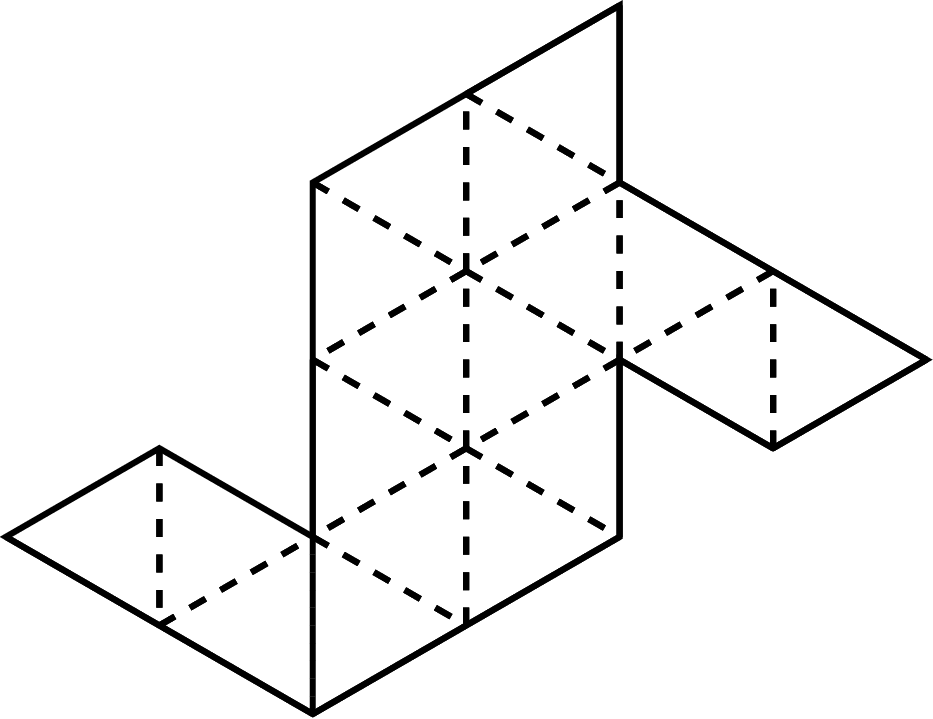}
            \subcaption{}
            \label{fig:arm1}
        \end{subfigure}
        \begin{subfigure}[b]{0.5\textwidth}
            \centering
            \includegraphics[width=5cm]{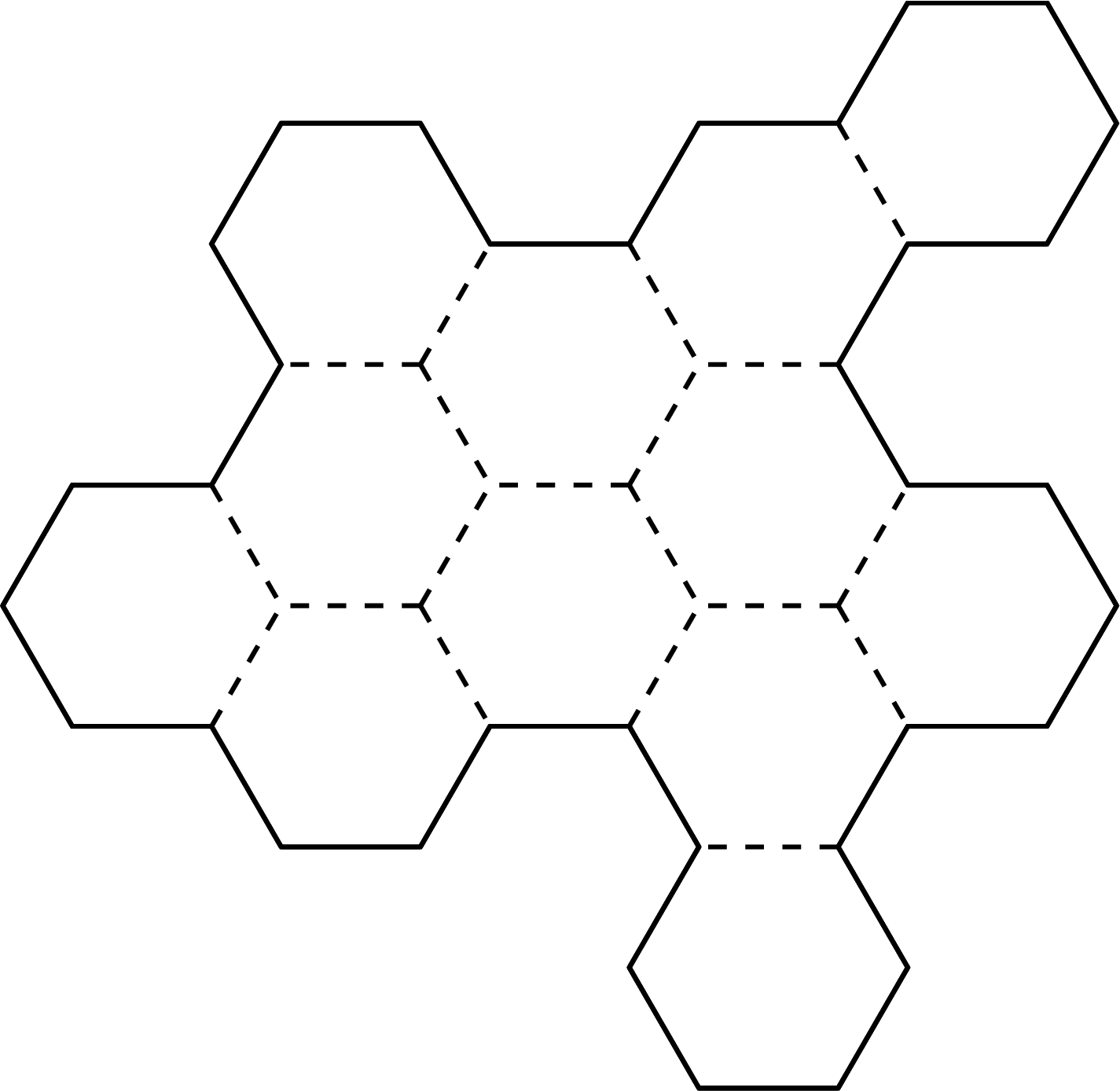}
            \subcaption{}
            \label{fig:arm2}
        \end{subfigure}
        \caption{Two tiles that exhibit a hierarchical structure to their tilings. Courtesy of J.S. Myers \cite{myerspriv}.}
        \label{fig myerspriv}
\end{figure}

\begin{figure}[h!]
\centering
\includegraphics[width=8cm]{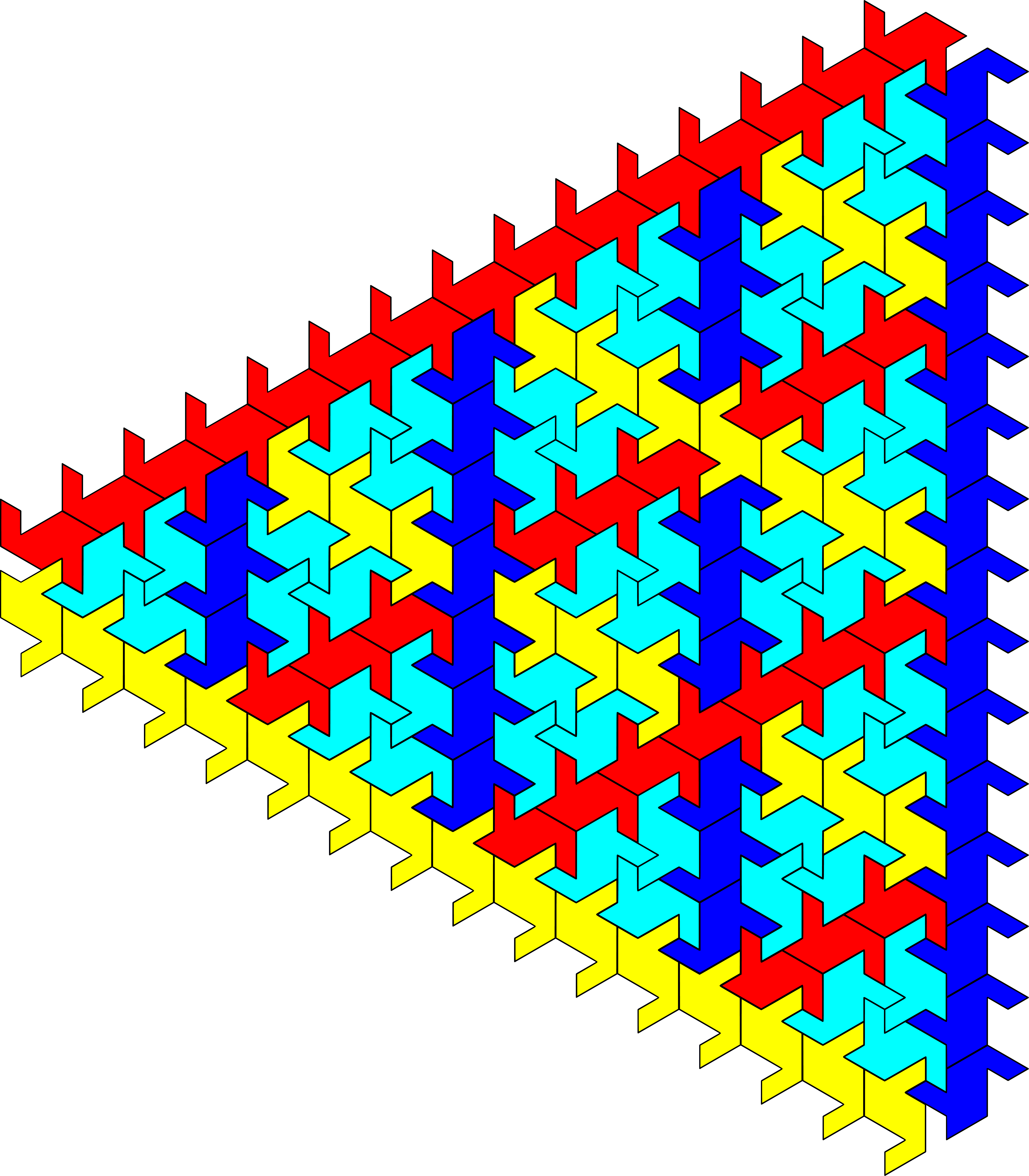}
\caption{ An "equilateral triangle" of "third order". To obtain triangles of larger order, one can extend this triangle as the smaller ones comprising it were, and in doing so we can make this triangle be any of the $4$ subtriangles of the next one, and also for the next and next, giving rise to $\mathfrak{c}$ many tilings. }
\label{fig hierarchy}
\end{figure}
The hierarchical structure exhibited in the previous example is common in aperiodic sets of tiles, but these examples are remarkable for we know of only one family of aperiodic tiles, all of which behave in an analogous manner to each other \cite{spectre, hat}.

Worthy of notice are some other ways in which a number of tilings (this time for a protoset in the first two references and in general in the rest) can be proven to be $\mathfrak{c}$. Kari \cite{kari} showed that to each real number $\alpha \in [\frac{1}{2}, 2]$ corresponds a tiling by his protoset of Wang tiles, thus giving the continuum many tilings. The way in which he proved this was by emulating the behavior of Mealy machines that multiply Beatty sequences of real numbers by rational coefficients. See also \cite{jeandelrao} for more. Dolbilin \cite{dobi}, Danzer and Dolbilin \cite{danzer} (see also \cite{senechal}, section 7.6.2 for Radin's unpublished proof) showed that $\textit{any}$ aperiodic protoset admits $\mathfrak{c}$ many tilings, a contraposition of which says that any $\sigma$-morphic protoset necessarily admits a periodic tiling! The way in which they prove this result is especially interesting since they do not rely on the hierarchical structure of tilings, but rather use measure theory or rely on some deeper structural results.

\begin{figure}[b!]
\centering
\includegraphics[angle=90,origin=c, height=10cm]{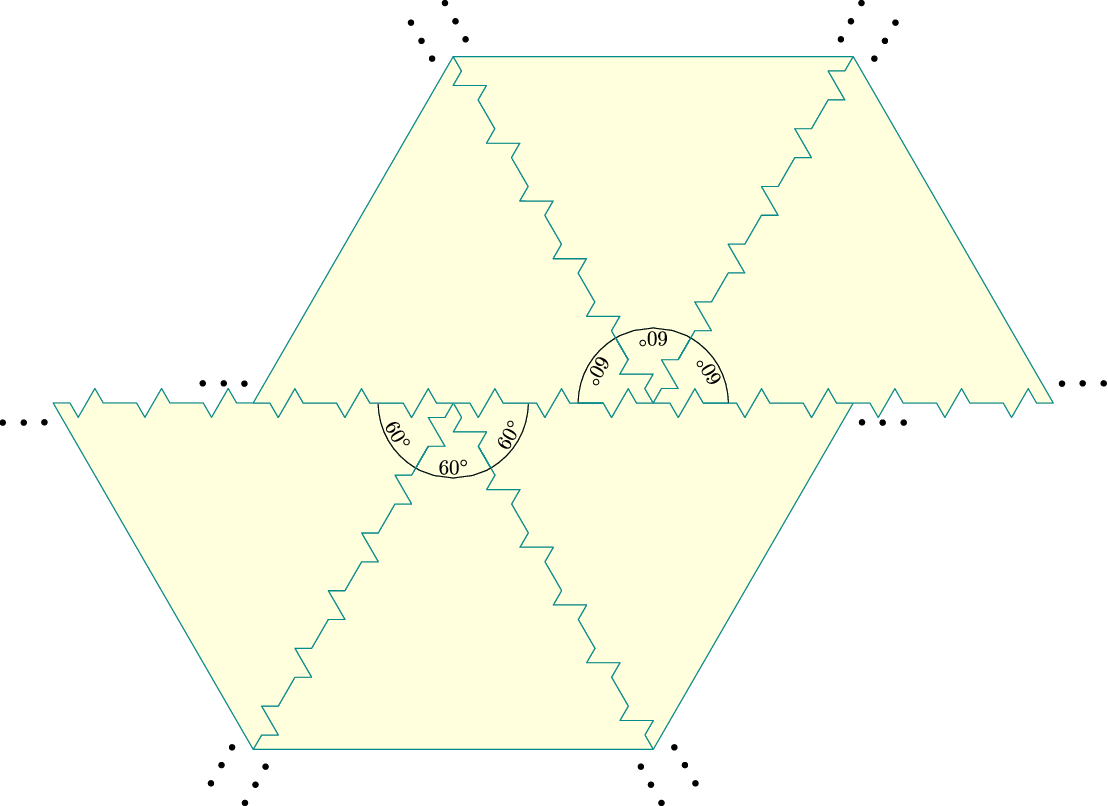}
\caption{One of $\aleph_0$ many distinct tilings by the infinite tile discussed in the text. All the other tilings are obtained by discretely shifting the three upper tiles relative to the three lower tiles in this figure.}
\label{fig: diftilings}
\end{figure}

\subsection{Different ways of obtaining $\aleph_0$ many tilings}

Next we survey the main ideas used in the literature of obtaining $\aleph_0$ tilings. One of the prevalent ideas of doing this is to enforce a larger structure to emerge which will then (as an unbounded tile, say) be $\sigma$-morphic. So far there has been essentially only one way of doing this, which is shown in Figure \ref{fig: diftilings}. In it, we are assuming that we are tiling the plane with the infinite sections of the plane, $6$ of which are shown. It then isn't hard to argue that the structure similar to the one shown in Figure \ref{fig: diftilings} is forced, with the only difference that the vertices at which the apexes of $3$ tiles meet are a different distance apart. This distance can in turn only be a natural number, due to the symmetric protrusions on the edges of the big tile, which are "discretizing" the boundary of the tile, intuitively speaking. And for each natural number $n$ we obtain a distinct tiling which leads to the conclusion that this infinite tile is $\sigma$-morphic. What is left to be done is to somehow force this bigger structure to emerge in a tiling, and in such a way that each tile is a part of at least one such infinite tile. Schmitt managed to do this with a protoset of two tiles, \cite{schmitttrougao}, and Ba\v si\' c, D\v zuklevski and Slivkov\' a \cite{mi} managed to do it with generalized matching rules (note that Schmitt's construction can also be translated into a single tile with generalized matching rules, as was done in \cite{mi}, section 5.2 for polymorphic tiles, but the construction presented in \cite{mi} is different).

\begin{figure}
        
        \begin{subfigure}[b]{0.5\textwidth}
            \centering
            \includegraphics[width=5cm]{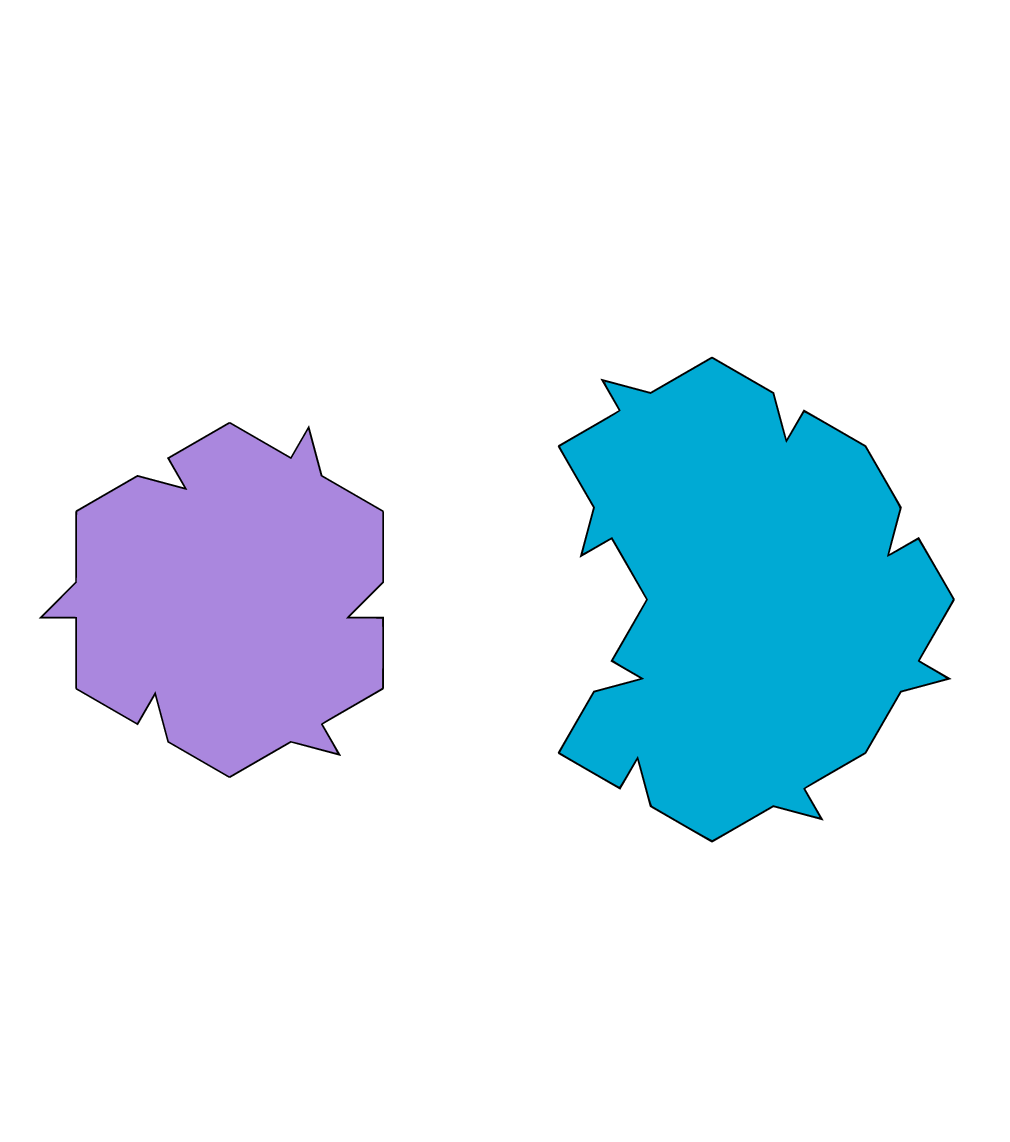}
            \subcaption{}
            \label{fig:arm1}
        \end{subfigure}
        \begin{subfigure}[b]{0.5\textwidth}
            \centering
            \includegraphics[width=6cm]{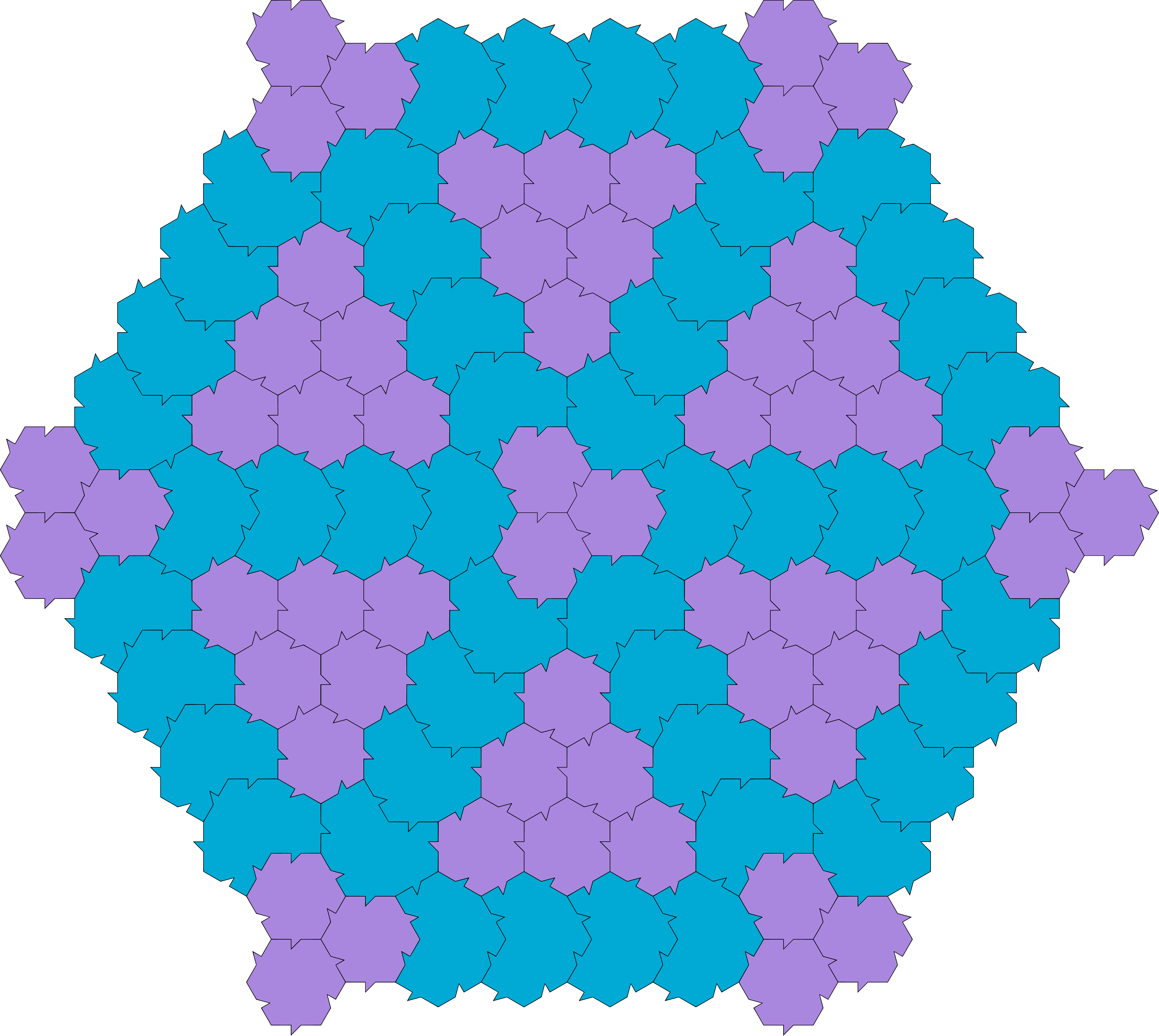}
            \subcaption{}
            \label{fig:arm2}
        \end{subfigure}
        \caption{(a) The $\sigma$-morphic protoset discovered by Schmitt \cite{schmitthierarchy}. (b) A patch from one of the tilings admitted by this protoset. All but a finite amount of "anomalous" tilings are obtained in a similar fashion, by taking hexagons like this one which are comprised of triangles containing $4,5,6,...$ purple tiles for their edges instead of $3$ as depicted here.}
        \label{fig schmitt2}
    \end{figure}

Another, brilliant, idea of Schmitt's is not to force $\aleph_0$ many tilings by relative positions of some larger structures in a tiling, but by creating tilings from larger and larger structures. To see what we mean by this, assume that we can construct a set of tiles that could join together to form a larger tile, say of area $P,$ which is monomorphic. Assume further that there is a different way in which we could join those tiles so as to obtain a tile similar to this one, but whose area is, say, twice as large. And assume further that we could do this for any integer $n$: for each integer $n$ we can join the tiles from the given protoset to obtain a tile similar to the original, monomorphic one, but whose areas are $nP$ (it's only important that the multiples are discrete, i.e. that the whole sequence is countable). If the shape is such that no two of these tilings are the same, we will obtain $\aleph_0$ many distinct tilings. The idea is to somehow force this to happen so that no other tilings are possible and the protoset is then $\sigma$-morphic. Schmitt indeed managed to do this \cite{schmitthierarchy} and one of the protosets he discovered is shown in Figure \ref{fig schmitt2} (a). The hexagonal structure in Figure \ref{fig schmitt2} (b) can be extended into a unique tiling, and we can form larger and larger structures akin to this one, but with triangles of purple tiles having $4,5,6,...$ tiles at their sides. It can be shown that there are a few "anomalous" tilings together with these, but the total amount is still countable.

\section{The main section}\label{sec4}

In this section we present a protoset that forces $\aleph_0$ many tilings in a way different to those described in the previous section, and also show how a $\sigma$-morphic protoset consisting of convex polygons can be constructed.

\subsection{A new $\sigma$-morphic protoset}\label{4.1}

\begin{thm}
    The protest shown in Figure \ref{fig p3} \textup{(a)} is $\sigma$-morphic.
\end{thm}

\begin{proof}
    Observe first that the lateral bumps and nicks of each of the three tiles force the tiles in the same row to all be the translates of each other. Also observe that the upper side of the purple tile forces all the tiles above it to be its translates. If the only tiles that appear in a tiling are purple ones, then the tiling is easily seen to be unique. Next, observe that the row of green tiles has to be followed (in the upward direction) by an infinite number of purple rows. The lower side of the green row can either be accommodated by the other green row (in which case the whole tiling is forced), or by a red row. In the latter, the same analysis holds since the lower sides of red and green tiles are the same. This gives rise to two options: either all the rows below the green one are red, or there are finitely many red followed by a green and infinitely many purple (but "pointing downwards" this time). It is trivially seen that we can have an arbitrary (finite) number of red layers in these tilings, giving rise to $\aleph_0$ many tilings. If no green tile is used, the only other option is the one in which only the red tiles are used, and there is only one such tiling. This completes the proof.
\end{proof}

It is worth noting that the author is at present not aware of any protoset of $2$ prototiles which can force the same structure to emerge. The reason why this is hard to construct is that one seems to need the second tile from Figure \ref{fig p3}(a) in order to make a transition from red to purple tiles. Whether this can be avoided in any way and the number of tiles taken down is an interesting open problem.

\begin{figure}
\begin{subfigure}{.5\textwidth}
  \centering
  \includegraphics[width=7.5cm]{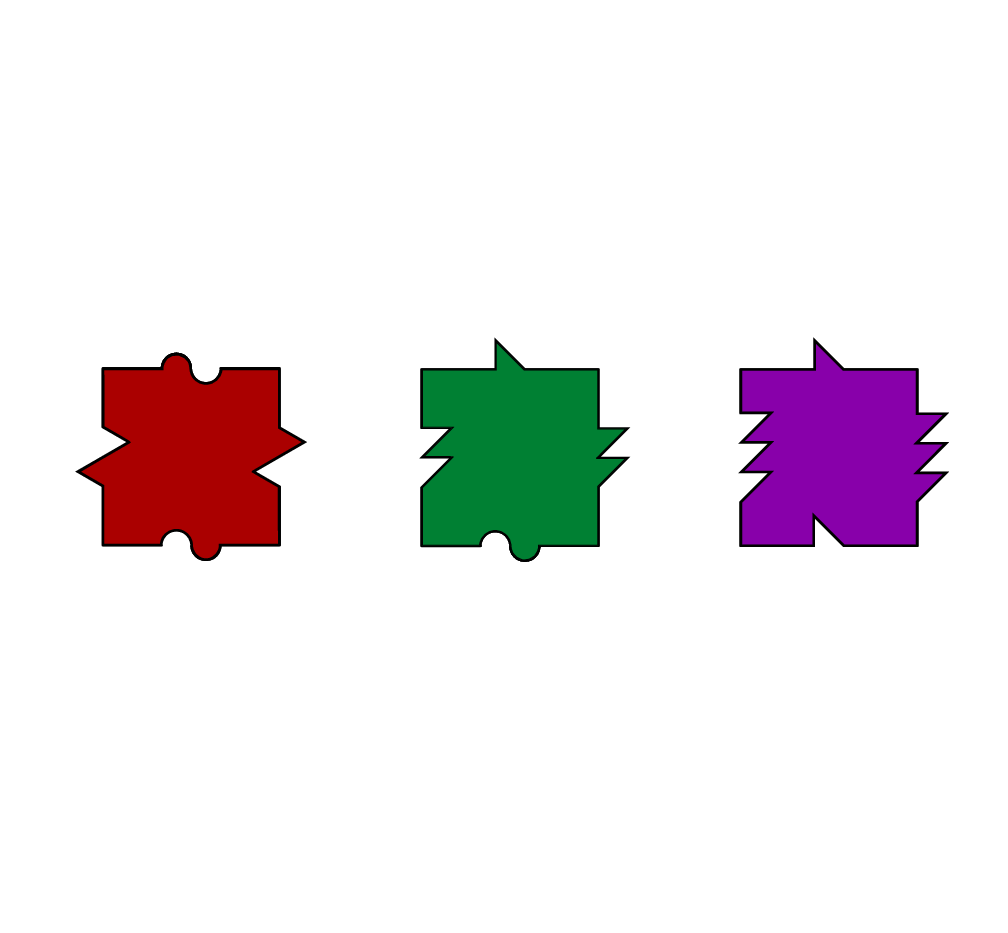}
  \caption{  }
  \label{fig:sfig1}
\end{subfigure}
\begin{subfigure}{.5\textwidth}
  \centering
  \includegraphics[width=7.5cm]{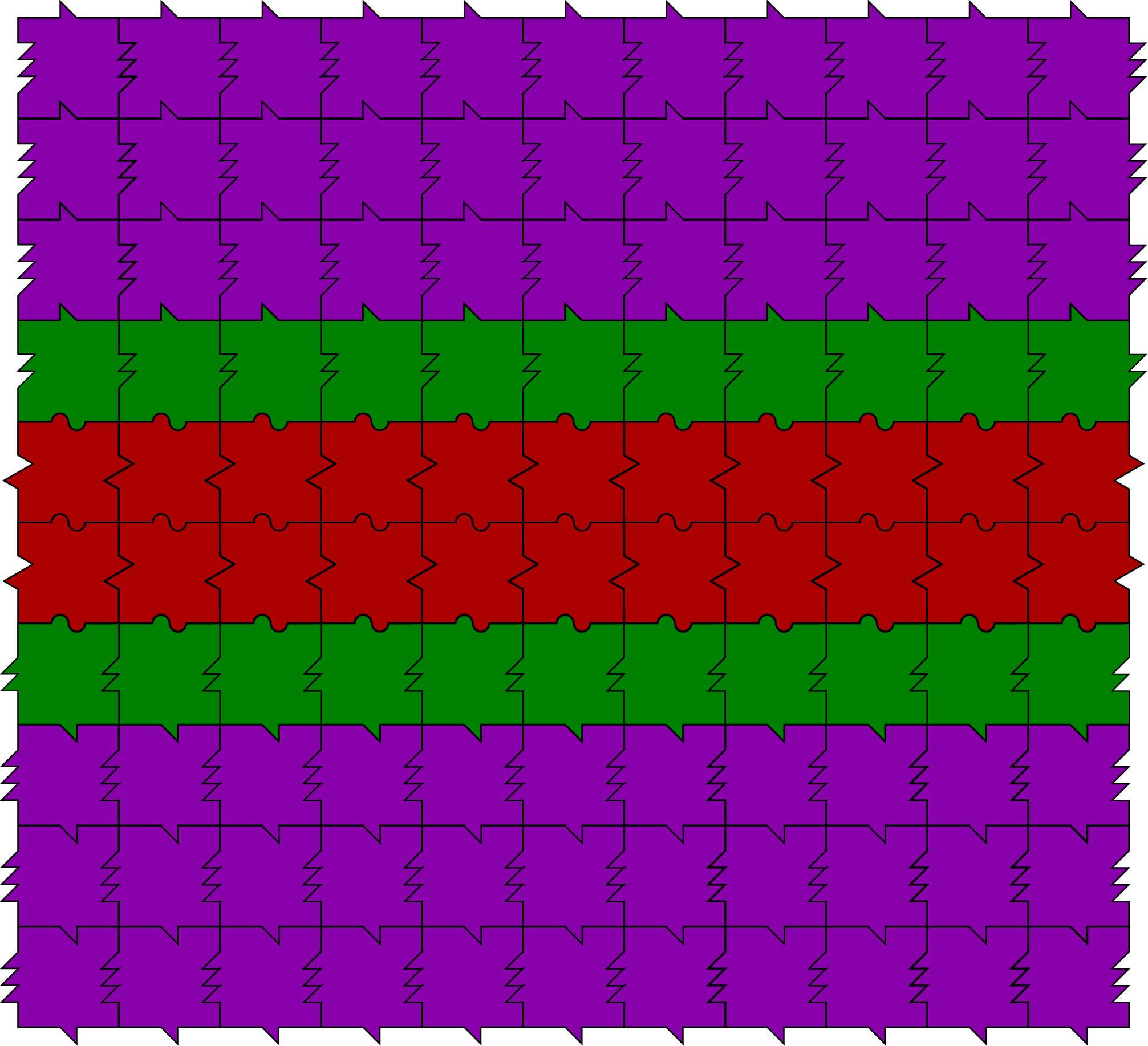}
  \caption{ }
  \label{fig:sfig2}
\end{subfigure}
\caption{(a) A protoset of $3$ prototiles that is $\sigma$-morphic. (b) One of the tilings admitted by this protoset. All but two tilings are obtained by varying the number of rows of the burgundy tiles, which can take a value of any non-negative integer.}
\label{fig:fig}
\label{fig p3}
\end{figure}

\subsection{Convex protoset}\label{sec4.2}
As we saw in Section \ref{sec2}, the main geometric tool for forcing the tiles to obey the matching rules that we put on them is by the means of bumps and nicks (c.f. for instance Figure \ref{fig schmitt2}). It is for this reason that restricting to convex tiles is an interesting variation that is often considered in the literature. For instance, a problem of the existence of a convex tile whose Heesch number is larger than $1$ yet finite is still open \cite{mannprez}. Results of Amman, Danzer, and Gr{\"u}nbaum and Shephard \cite{TIP}, page 549, are, however, much more in line with what we are trying to achieve. They managed to construct sets of tiles which have $3,$ $9$ and $6$ convex polygons, respectively, which are aperiodic. What's more, they were obtained by \textit{recompositions} of various versions of Penrose tiles \cite{penrose} and behave in exactly the same way. Following $\cite{TIP}$, we say that a protoset $\mathcal{T}_1$ is obtained by \textit{recomposition} from the protoset $\mathcal{T}_2$ if it is possible to mark the tiles of $\mathcal{T}_1$ in such a way that in \textit{any} tiling $T_1$ admitted by $\mathcal{T}_1$ the marks on the tiles are the vertices and edges of \textit{some} tiling admitted by $\mathcal{T}_2$.  The prototset from Figure \ref{convex} is obtained by recomposition of the protoset from Figure \ref{fig schmitt2}. We continue the discussion in details below.

Before that we note that a convex tile that is $\sigma$-morphic cannot exist. Indeed, each triangle and quadrangle tile the plane and it is easy to see that each triangle is $\mathfrak{c}$-morphic while quadrangles can either be monomorphic or $\mathfrak{c}$-morphic as well. The story of the discovery of the fifteen classes of convex pentagons that tile the plane is an interesting one, culminating in the discovery by Mann, McLoud-Mann and Von Derau \cite{mannpentagon} of the fifteenth and supposedly final such class (Rao \cite{rao} has announced a computer-assisted proof of this claim in 2017 but it has not yet been peer-reviewed and published; the consensus of the community is, however, that the list is now complete); see \cite{mannpentagon, schatesneider} and references therein for the rich history of the discoveries of the fifteen classes. It is again not hard to see that each pentagon belonging to some of those classes is either polymorphic or $\mathfrak{c}$-morphic. 
There are also only three families of convex hexagons that tile the plane \cite{adams, ReinhardtPhD} and all of them are monomorphic. Finally, it is a well known theorem that no convex $n$-gon for $n\geqslant 7$ tiles the plane \cite{TIP, niven, Reinhardtconv}.

We return back to the announced theorem:

\begin{thm}
There exists a protoset $\mathcal{P}$ consisting of convex prototiles that is $\sigma$-morphic.
\end{thm}

\begin{proof}

We will show that the protoset consisting of distinct polygons shown in Figure \ref{convex} has the desired property. The main idea is to choose the angles and sides of the convex polygons in such a way that the only way in which any particular vertex of any of those tiles could be surrounded would be to force the configurations from Figure \ref{convex}. For instance, the angles of the lowermost three tiles into which the right tile from Figure \ref{convex} is divided are such that the only way in which the angle of tile $4$ that is incident to the vertex $Q$ can be surrounded fully is in the way that is shown in the figure.

Since we are aiming to construct the smallest protoset with this property, we will recompose tiles from the Schmitt's protoset (Figure \ref{fig schmitt2} (a)) rather than our own from the previous subsection (though we will show how it can be done). The first step is to modify the bumps and nicks present at the original tiles in a manner shown in Figure \ref{convex}. To do this, we exchange the edge with a bump on it with the two edges that are making an angle $\alpha$ above that original edge and that are proturbing away from it, and similarly for the nicks; cf. Figure \ref{convex}. We have to make sure that they are still asymmetric and that the bump can only be accommodated by a nick and vice versa. The way we do this is to take the angle $\alpha$ of the new bumps and nicks to be unique among the angles (and their sums) of tiles in $\mathcal{P}$. We will have some additional requirements on $\alpha$ that will come from the forcing of the desired configurations that we mentioned, but we will focus on those later.

In the following, we will ensure that no two consecutive edges of the obtained tiles, that are not the edges of the original tiles (but with modified bumps and nicks), are of the same length (with the exception of tile 1, which we will consider separately). Also, we don't want any number of angles to sum up to $\theta$ or $180^{\circ}-\theta$, for any inner angle $\theta$ of any of the tiles in $\mathcal{P}$. The reason for introducing these constraints is that to prove that the configurations around vertices will be forced to be equal to those shown in Figure \ref{convex} as long as at least one tile is present, we have to prove that the right tiles (with adequate angles at the vertex) will be present around the vertex, that they will be in the same order around the vertex as shown, and that no tile will be of the opposite orientation to that shown in Figure \ref{convex} (i.e. flipped). We first have to make sure that the correct tiles appear around the vertex, which we do either by letting the angles be such that they only partake in one sum that equates to $360^{\circ}$, or by a bit more complicated argument, but we'll leave that for later. Having forced the correct tiles to appear around each vertex, to see why the order of the tiles will have to coincide with that shown in Figure \ref{convex}, take again for example the vertex $Q$ from Figure \ref{convex}. The same analysis will work for all the vertices in the interior of the original tiles (which are the only ones with which we are concerned) since they are all $3$-valent. Notice that the edge $QS$ is shorter than the edge $QT$, and notice also that a similar statement holds for other vertices since no two consecutive inner edges of any tile are of equal length. Therefore, among the edges $QR, QS$ and $QT$ there are no two of the same length (for otherwise some tile would have two consecutive edges of the same length). Assuming now the contrary, that the tiles could be in some other order around this vertex (the other one since there are $2$ in total once the tiles that participate in the surrounding are fixed), we observe that the tile $6$ has to touch tile $5$ along its edge $QR$, since otherwise the edges of tiles $5$ and $6$ that are touching would be of different length. But since $QS,QT$ and $QR$ are all of different lengths, we would have an angle equal to $180^{\circ}-\theta$ present, which cannot be filled due to our choice of angles. Now, once the order of the tiles is forced, we have to show that each tile will be in the exact orientation as shown in Figure \ref{convex}. But this easily follows from the lengths of consecutive inner edges being different and the requirement on angles, using an analogous argument as before.

Let us now describe the way in which we ensure these two conditions (on the angles and lengths of edges). We shall first describe the way in which the convex polygons are obtained and then show that they indeed fulfill the two conditions. Focusing on the left original tile first, we divide it by connecting some of the vertices of the original tile with the center of the red hexagon, cf. Figure \ref{convex}. The angle $\alpha$ of the bumps and nicks is chosen such that the only sum that equates to $360^{\circ}$ and that contains $\alpha, \beta$ or $\gamma$ in it is exactly $\alpha + \beta + \gamma  = 360^{\circ}.$ Also, the edges $ZW$ and $ZZ'$ have to be of different lengths (and note here that the position of vertex $Z$ relative to the vertices of the original tile -- including $W$ -- is determined by the angle $\alpha$), which we will show is possible to do later. For now, we suppose that $\alpha$ has been chosen and proceed in our discussion, and will retroactively explain all the details. As for dividing the right tile from Figure \ref{convex}, we connect points $W$ and $Z$ and we place two points $Q$ and $R$ such that the angles that the vertices make with points $S,T$ and $U,V$, respectively, are unique among the angles, fulfill the criterion, can only sum up to $360^{\circ}$ if they are all used, and are such that the tiles obtained by connecting $Q$ to $S$ and $T$ and $R$ to $U$ and $V$ are convex, and no two consecutive sides which they do not share with the original tile are of equal length. We will note later that the angles $\angle Z'ZW$ and $\angle Z''ZW$ are determined by the choice of $\alpha$ so we will consider that they were fixed prior to the placement of $QR$ and "tweak" the angles at $Q$ and $R$ accordingly. That this can be achieved is again fairly obvious since there are continuum many placements of the segment $QR$ that fulfill the convexity criterion (which must lie in the intersection of some halfplanes) and taking one that will fulfill the requirements on the angles is then almost inevitable since the set of forbidden angles is finite. With that the construction of the prototiles of $\mathcal{P}$ is complete.

Next we show that indeed the prototiles from $\mathcal{P}$ can be surrounded only in the way they are surrounded in Figure \ref{convex}. Let us start from tile $1$. Tile $1$ is different from the others in that it actually has two consecutive edges it does not share with the original tile that are of the same length. This, however, will not be important since the forcing of the placement of tiles will happen along other angles (for instance, angle $\beta$). Notice that the angle $\beta$ of each tile can be surrounded only with angles $\alpha$ and $\gamma$ and exactly in that order and orientation (note here that several tiles have angles equal to $\alpha$ and that therefore this surrounding is not unique; but each tile that has angle $\alpha$ as one of its inner angles will have the two edges incident to it of the same length as all the other such tiles and then the orientation of the tile makes sense). In particular, although, as mentioned, this will not entirely force a unique surrounding, it will force the placement of the tile $2$ as shown in Figure \ref{convex}. This tile will in turn force the placement of the tile $3$ from the same figure and the desired surrounding is obtained. Arguing in a similar manner, we see that any of the tiles $4,5$ and $6$ will force the rest through the angles at $Q$, and similarly any of tiles $5,6$ and $7$ will force the rest along the angles at vertex $R$. We are left to show that the tiles $7$ and $8$ force the placement of each other. If the angles at $Z$ are distinct from those at $P$ we are done. Otherwise (even though $\alpha$ could be chosen so that this is not the case (which is a bit harder to prove than the current statement)), the angles have to be equal. It then suffices to show that the edge $ZW$ is of different length than any of the edges $AP, BP$ and $CP$. To show this we will show that $|WZ| > |Z'Z''|$. Note that the point $W$ is obtained by rotating $Z'$ around $W'$ for $90^{\circ}$ in the clockwise direction, while point $Z'''$ is obtained by rotating point $Z'$ around $W'$ by $60^{\circ}$ in the same direction, meaning that $W$ is outside the triangle $Z'Z''Z'''$. Furthermore, we claim that $\angle Z'Z'''W$ is obtuse. To see why, observe that $\angle W'Z'''W$ is equal to $75^{\circ}$ and that $\angle W'Z'''Z' = 60^{\circ}$, making $\angle Z'Z'''W$ equal to $135^{\circ}$. This means that $|Z'W| > |Z'Z'''|$ and so the complement of the circle centered at $W$ and of radius $Z'Z'''$ with respect to the triangle $Z'Z''Z'''$ is nonempty and thus the original point $Z$ could have been chosen in that region, making sure that $ZW$ is longer than $Z'Z'''$ which is in turn longer than both $ZZ'$ and $ZZ''$ since they are contained in the triangle $Z'Z''Z'''$. As $Z'Z \cong AP$ and $Z''Z \cong BP,$ and as $\triangle ABC \cong \triangle Z'Z''Z''',$ we have shown that $WZ$ is longer than all of $AP, BP$ and $CP.$ We next proceed to use the argument we have used when discussing the orientations of the tiles to show that indeed tiles $7$ and $8$ have to meet as is shown in Figure \ref{convex}.

As now the protoset $\mathcal{P}$ behaves in the same way that the pair of prototiles discovered by Schmitt from Figure \ref{fig schmitt2} we have shown that $\mathcal{P}$ is indeed $\sigma$-morphic.

\end{proof}

Note that the left original prototile had to be divided in a symmetric manner: were it otherwise we could at each place in the tiling choose a different orientation of the tile, which would lead to $\mathfrak{c}$ many tilings.

\begin{figure}[h!]
\centering
\includegraphics[width=14cm]{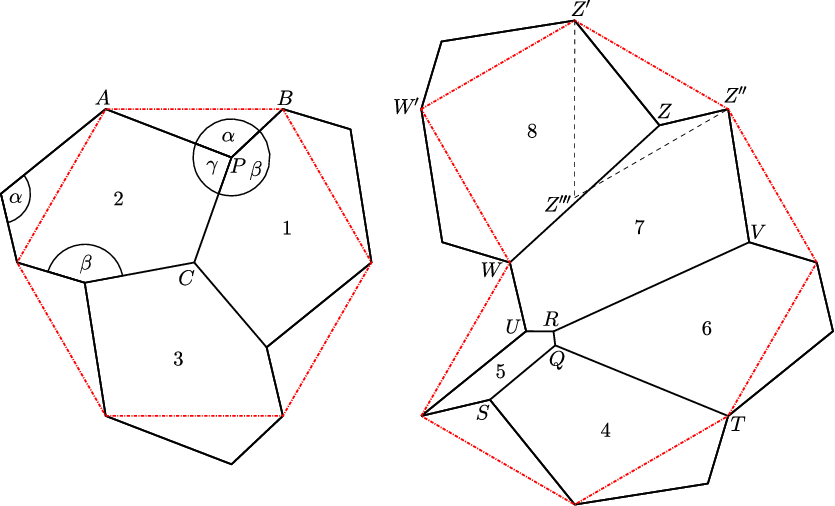}
\caption{ Prototiles of the protoset $\mathcal{P}$. There are $4$ distinct convex hexagons in $\mathcal{P}$ and $2$ distinct convex pentagons. The original protoset of Schmitt that we are subdividing to obtain $\mathcal{P}$ is shown in red (without the bumps and nicks) and is referred to as the original protoset in the proof. }
\label{convex}
\end{figure}

We also note that the protoset from Section \ref{4.1} can be recomposed into convex prototiles, as is shown in Figure \ref{fig protoconv3}. We will, however, omit the proof as it uses similar ideas as the previous one, and only note that some of the prototiles are quadrilaterals (5 out of 11) and that they tile the plane by themselves (it is for this reason that triangles are not present in any of the protosets for they alone admit $\mathfrak{c}$ many tilings). The angles of these quadrilaterals can, however, be chosen such that they are all monomorphic and thus no problem will sprout from them. It is also necessary to divide the leftmost prototile from Figure \ref{fig protoconv3} in a centrally-symmetric manner, which was indeed done in the said Figure. 

\begin{figure}[h!]
\centering
\includegraphics[width=14cm]{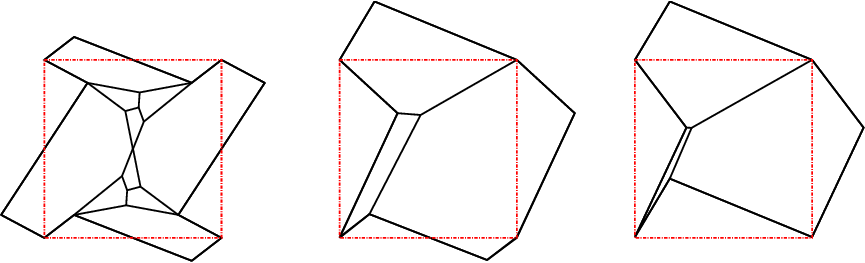}
\caption{ The convex protoset obtained by recomposition of prototiles from the protoset shown in Figure \ref{fig p3}. It consists of $5$ quadrilaterals, $4$ pentagons a hexagon and a heptagon.}
\label{fig protoconv3}
\end{figure}

\section{Open problems}\label{sec5}

For the end, we would like to draw the attention of the reader towards some open problems regarding the $\sigma$-morphic tiles. The original problem is of course still open and we don't even know what happens if the following weakening of the conditions is added.

\begin{problem}
    Does there exist a disconnected tile that is $\sigma$-morphic?
\end{problem}

The problem seems to be quite hard and is the only variation among the common variations that the authors of \cite{mi} weren't able to solve. 

Speaking of the said paper, the authors managed to construct a $\sigma$-morphic tile in $\mathbb{E}^d$ for $d \geqslant 3.$ The solid that they construct, however, has a lot of holes: the three-dimensional tile has genus 48! This, in some hand-wavy sense, allows a tile to influence tiles that are not so near to it, allowing for non-local matching rules. Recall from the discussion in Section \ref{sec2.1} that "having only local influence" can lead to $\mathfrak{c}$ many tilings, and limiting the bodies to those of genus $0$ we are in a situation in which these bodies are not so free to influence tiles from "different layers". All this is rather hand-wavy but we hope that the intuition is clear to the reader. Having that in mind, we ask the following question:

\begin{problem}
    Does there exist a $\sigma$-morphic solid in $E^d,$ $d\geqslant3$ whose genus is $0$?
\end{problem}

Returning back to the topic of Section \ref{sec2.1}, we ask the following, somewhat non-well defined question:

\begin{problem}
    What other ways are there of achieving $\mathfrak{c}$ many tilings by a protoset?
\end{problem}

We recall that Ammann constructed a set of three convex polygons that are aperiodic \cite{TIP}, page 549. This motivates us to ask the following:

\begin{problem}
    Does there exist a set of three (or even two) convex prototiles that is $\sigma$-morphic?
\end{problem}

Finally, we ask if the axioms of the set theory could potentially have any say in the matter of the original problem. Note first that a tile cannot tile the plane in more than $\mathfrak{c}$ many distinct ways. Indeed, in each tiling, we have $\aleph_0$ many tiles, and each of the tiles can be at any position in the plane (which we can describe by a triplet $(x,y,\alpha)$, where $(x,y)$ are the coordinates of one particular vertex of tile $T$ (which has been chosen in advance, and is the same for all $T_i$) and where the angle $\alpha$ is the angle that a fixed edge incident with the fixed vertex from before creates with the $x$--axis; ultimately, there are $\mathfrak{c}$ many positions for each tile). This is of course a gross overestimate, but still, this gives us that each of $\aleph_0$ many tiles has $\mathfrak{c}$ many choices, and ultimately, that the number of tilings does not exceed $\mathfrak{c}^{\aleph_0} = \mathfrak{c}$. This gives rise to an interesting question of the existence of tiles that tile the plane in $\kappa$ many distinct ways, where $\aleph_0 < \kappa < \mathfrak{c}$. Of course, for those $\kappa$ to exist in the first place, one has to work in the adequate axiomatic system. That a choice of the axioms of the set theory that we choose to work with could influence the result of a problem in combinatorial geometry, could perhaps be motivated by the results of Sz\' ekely, Shelah and Soifer and more recently, Payne \cite{payne, soifer, szekely}. We thus ask our final problem:

\begin{problem}
    Do there exist tiles that are $\kappa$-morphic, for $\aleph_0 < \kappa < \mathfrak{c}$ assuming that we are working in the axiomatic systems in which such cardinalities exist?
\end{problem}

\section*{Acknowledgments}
The author would like to thank Joseph S. Myers for allowing him to showcase the tiles from Section \ref{sec2.1}, and also for the helpful references for the same section. The author would also like to thank Toma\v s Valla for attracting his attention to the problem discussed in this paper.

\end{document}